\documentclass[preprint]{imsart}

\usepackage[margin=1.2in]{geometry}
\usepackage{amsfonts,amsmath,amsthm}
\usepackage[latin1]{inputenc}
\usepackage{textcomp}
  \usepackage{version,amsfonts,amsmath,boxedminipage,alltt}
  \usepackage{graphics}
 \usepackage{graphicx}

\newtheorem{theorem}{Theorem}
\newtheorem{lemma}{Lemma}
\theoremstyle{remark}
\newtheorem{remark}{Remark}

\let\epsilon=\varepsilon

\let\phi=\varphi

 \let\Th=\Theta

\newcommand{\E}{\field{E}}

\newcommand{\field}[1]{\mathbb{#1}}
\newcommand{\R}{\field{R}}

\newcommand{\e}{{\mathcal E}}

\newcommand{\fracn}{\mbox{$\frac{1}{n}$}}
\def\der^#1_#2{\frac{\partial^{#1}}{\partial {#2}^{#1}}}

\RequirePackage[OT1]{fontenc}

\RequirePackage{amsthm,amsmath}
\def\1{\mbox{1\hspace{-.25em}I}}

\def\t{\theta}
\def\d{\delta}

\def\Th{\Theta}
\def\e{\varepsilon}
\def\b{\beta}

\def\bR{\mathbb{R}}

\startlocaldefs
\endlocaldefs

\begin{document}

\begin{frontmatter}
\title{An $\{\ell_1,\ell_2,\ell_\infty\}$-Regularization Approach to High-Dimensional Errors-in-variables Models}
\runtitle{$\{\ell_1,\ell_2,\ell_\infty\}$-Regularization to EIV Models}
%\thankstext{T1}{Footnote to the title with the `thankstext' command.}

\begin{aug}
\author{\fnms{Alexandre} \snm{Belloni}\ead[label=e1]{abn5@duke.edu}}

\address{The Fuqua School of Business\\
Duke University\\
\printead{e1}}

\author{\fnms{Mathieu} \snm{Rosenbaum}
\ead[label=e2]{mathieu.rosenbaum@upmc.fr}
%\ead[label=u1,url]{www.foo.com}
}

\address{ Laboratoire de Probabilit\'es et Mod\`eles Al\'eatoires,\\
Universit\'e Pierre et Marie Curie (Paris 6), and \\
Centre de Recherche en Economie et Statistique,\\
ENSAE-Paris Tech\\
\printead{e2}}

\author{\fnms{Alexandre B.} \snm{Tsybakov}
\ead[label=e3]{alexandre.tsybakov@ensae.fr}
%\ead[label=u1,url]{www.foo.com}
}

\address{ Centre de Recherche en Economie et Statistique,\\
ENSAE-Paris Tech\\
\printead{e3}}

%\thankstext{t1}{Some comment}
%\thankstext{t2}{First supporter of the project}
%\thankstext{t3}{Second supporter of the project}
\runauthor{Belloni, Rosenbaum and Tsybakov}

\affiliation{Some University and Another University}

\end{aug}

\begin{abstract}
Several new estimation methods have been recently proposed for the linear regression model with observation error in the design. Different assumptions on the data generating process have motivated different estimators and analysis. In particular, the literature considered (1) observation errors in the design uniformly bounded by some $\bar \delta$, and (2) zero mean independent observation errors. Under the first assumption, the rates of convergence of the proposed estimators depend explicitly on $\bar \delta$, while the second assumption has been applied when an estimator for the second moment of the observational error is available. This work proposes and studies two new estimators which, compared to other procedures for regression models with errors in the design, exploit an additional $\ell_\infty$-norm regularization. The first estimator is applicable when both (1) and (2) hold but does not require an estimator for the second moment of the observational error. The second estimator is applicable under (2) and requires an estimator for the second moment of the observation error. Importantly, we impose no assumption on the accuracy of this pilot estimator, in contrast to the previously known procedures.  As the recent proposals, we allow the number of covariates to be much larger than the sample size. We establish the rates of convergence of the estimators and compare them with the bounds obtained for related estimators in the literature. These comparisons show interesting insights on the interplay of the assumptions and the achievable rates of convergence.
\end{abstract}

%\begin{keyword}[class=MSC]
%\kwd[Primary ]{}
%\kwd{}
%\kwd[; secondary ]{}
%\end{keyword}

%\begin{keyword}
%\kwd{}
%\kwd{}
%\end{keyword}

% history:
% \received{\smonth{1} \syear{0000}}

%\tableofcontents

\end{frontmatter}

\section{Introduction}
\noindent Several new estimation methods have been recently proposed for the linear regression model with observation error in the design. Such problems arise in a variety of applications, see \cite{RT1,LW,SFT1,SFT2}. In this work we consider the following regression model with observation error in the design:
\begin{eqnarray*}
y&=& X\theta^* + \xi,\\
Z&=&X+W.
\end{eqnarray*}
Here the random vector $y\in\mathbb{R}^n$ and the random $n\times p$ matrix
$Z$ are observed, the $n\times p$ matrix $X$ is unknown, $W$ is an
$n\times p$ random noise matrix, and $\xi\in \mathbb{R}^n$ is a random noise
vector. The vector of unknown parameters of interest is $\theta^*$ which is assumed to belong to a given convex subset $\Theta$ of $\mathbb{R}^p$ characterizing some prior
knowledge about $\theta^*$ (potentially $\Theta = \mathbb{R}^p$). Similarly to the recent literature on this topic, we consider the setting where the dimension $p$ can be much larger than the sample size $n$ and the vector $\theta^*$ is $s$-sparse, which means that it has not more than $s$ non-zero components.
\\

\noindent The need for new estimators under errors in the design arises from the fact that standard estimators (e.g. Lasso and Dantzig selector) might become unstable, see \cite{RT1}. To deal with this framework, various assumptions have been considered, leading to different estimators.
\\

\noindent A classical assumption in the literature is a uniform boundedness condition on the errors in the design, namely,
\begin{equation}\label{Wbound}
|W|_\infty \leq \bar \delta \ \ \mbox{almost surely},
\end{equation} where $|\cdot|_{q}$ denotes the $\ell_q$-norm for $1\le q\leq\infty$.
\noindent Note that this assumption allows for various dependences between the errors in the design. In this setting, the  Matrix Uncertainty selector (MU selector),
which is robust to the presence of errors in the design, is proposed in \cite{RT1}. The MU selector $\hat \theta^{MU}$ is defined as a
solution of the minimization problem
\begin{equation}\label{Def:MU}
\min\{ |\theta|_1: \,\, \theta\in\Theta, \,
\big|\fracn Z^T(y-Z\theta)\big|_\infty\leq \mu |\theta|_1 +\tau\},
\end{equation}
where the parameters $\mu$ and $\tau$
depend on the level of the noises of $W$ and $\xi$ respectively. Under appropriate choices of these parameters and suitable assumptions on $X$, it was shown in \cite{RT1} that with probability close to 1,
\begin{equation}\label{rtMU}
|\hat\theta^{MU}-\t^*|_q\le C s^{1/q}\{\bar\delta+\bar\delta^2\}|\t^*|_1+ C s^{1/q}\sqrt{\frac{\log p}{n}}, \quad  1\le q\le \infty.
\end{equation}
Here and in what follows we denote by the same symbol $C$ different positive constants that do not depend on %the parameters of the problem
$\theta^*$, $s$, $n$, $p$, $\bar\delta$.
The result (\ref{rtMU}) implies consistency as the sample size $n$ tends to infinity provided that the error in the design goes to zero sufficiently fast to offset $s^{1/q}|\t^*|_1$, and the number of variables $p$ and the sparsity $s$ of $\t^*$ do not grow too fast relative to the sample size $n$.
\\

\noindent An alternative assumption considered in the literature is that the entries of the random matrix $W$ are independent with zero mean, the values
\begin{equation*} \sigma_j^2=\frac{1}{n}\sum_{i=1}^n \mathbb{E}(W_{ij}^2), \ j=1,\ldots,p,\end{equation*}
are finite, and data-driven estimators $\hat{\sigma}_j^2$ of $\sigma_j^2$ are available converging with an appropriate rate. This assumption motivated the idea to compensate
the bias of using the observable $Z^TZ$ instead of the unobservable $X^TX$ in \eqref{Def:MU} thanks to the estimates of $\sigma_j^2$. This compensated MU selector, introduced in \cite{RT2} and denoted as $\hat\theta^{cMU}$, is defined as a solution of the
minimization problem
\begin{equation*}
\min\{|\theta|_1: \,\, \theta\in\Theta, \,
\big|\fracn Z^T(y-Z\theta)+\widehat{D}\theta\big|_\infty\leq\mu|\theta|_1
+\tau\},
\end{equation*}
where $\widehat{D}$ is the diagonal matrix with entries $\hat \sigma_j^2$ and
$\mu> 0$ and $\tau > 0$ are constants chosen according to the level of the noises and the accuracy of the $\hat\sigma_j^2$.\\

\noindent Rates of convergence of the compensated MU selector were established in \cite{RT2}. Importantly, the compensated MU selector can be consistent as the sample size $n$ increases even if the error in the design does not vanish. This is in contrast to the case of the MU selector, where the bounds are small
only if the bound on the design error $\bar \delta$ is small. In particular, under regularity conditions, when $\t^*$ is $s$-sparse, it is shown in \cite{RT2} that with probability close to 1
\begin{equation}\label{rt13}
|\hat\theta^{cMU}-\t^*|_q\le C s^{1/q}\sqrt{\frac{\log p}{n}}(|\t^*|_1+1), \quad  1\le q\le \infty.
\end{equation}

\noindent Under the same alternative assumption, a conic programming based estimator $\hat\theta^{C}$ has been recently proposed and analyzed in \cite{BRT2014}. The estimator $\hat\theta^{C}$ is defined as the first component of any solution of the optimization problem
\begin{equation}\label{conic}
\underset{(\theta,t)\in \bR^p\times \bR_+}{\min}\{|\theta|_1+ \lambda t : \,\, \theta\in\Theta, \,
\big|\fracn Z^T(y-Z\theta)+\widehat{D}\theta\big|_\infty\leq\mu t
+\tau, \ |\theta|_2 \le t\},
\end{equation}
where $\lambda$, $\mu$ and $\tau$ are some positive tuning constants. Akin to $\hat\theta^{cMU}$, this estimator compensates for the bias by using the estimators $\hat \sigma_j^2$ of $\sigma_j^2$. However it exploits a combination of $\ell_1$ and $\ell_2$-norm regularization to be more adaptive.  It was shown to attain a bound as in (\ref{rt13}) and to be computationally feasible since it is cast as a tractable convex optimization problem (a second order cone programming problem). Moreover, under mild additional conditions, with probability close to 1, the estimator (\ref{conic}) achieves improved bounds of the form
\begin{equation}\label{6}
|\hat \theta^C-\t^*|_q\le C s^{1/{q}}\sqrt{\frac{\log p}{n}}(|\t^*|_{2}+1), \quad  1\le q\le \infty,
\end{equation}
provided that $\hat D$ converges to $D$ in sup-norm with the rate
$\sqrt{(\log p)/n}$. 
It is shown in \cite{BRT2014} that the rate of convergence in \eqref{6} is minimax optimal in the considered model.\\

\noindent There have been other approaches to the errors-in-variables model, usually exploiting some knowledge about the vector $\t^*$, see \cite{LW,SFT1,CChen,CChen1}. Assuming $|\t^*|_1$ is known, \cite{LW} proposed an estimator $\hat \t'$ defined as the solution of a non-convex program which can be well approximated by an iterative relaxation procedure. In the case where the entries of the regression matrix $X$ are zero-mean subgaussian and $\t^*$ is $s$-sparse, under appropriate assumptions, it is shown in \cite{LW} that for the error in $\ell_2$-norm ($q=2$),
\begin{equation}\label{lw12}
|\hat \t'-\t^*|_2\le C(\theta^*) s^{1/{2}}\sqrt{\frac{\log p}{n}}(|\t^*|_{2}+1),
\end{equation}
with probability close to 1. Here, the value $C(\theta^*)$ depends on $\theta^*$, so that there is no guarantee that the estimator attains the optimal bound as in \eqref{6}. Assuming that the sparsity $s$ of $\t^*$ is known and the non-zero components of $\t^*$ are separated from zero in the way that
$$|\t^*_j| \ge C \sqrt{\frac{\log p}{n}}(|\t^*|_2+1),$$
an orthogonal matching pursuit algorithm to estimate $\theta^*$ is introduced in \cite{CChen,CChen1}. Focusing as in \cite{LW} on the particular case where the entries of the regression matrix $X$ are zero-mean subgaussian, it is shown in \cite{CChen,CChen1} that this last estimator satisfies a bound analogous to (\ref{6}), as well as a consistent support recovery result. \\

\noindent The main purpose of this work is to show that an additional regularization term based on the $\ell_\infty$-norm leads to improved rates of convergence in several situations.  We propose two new estimators for $\t^*$. The first proposal is applicable under a new combination of the assumptions mentioned above. Namely, we assume that the components of the errors in the design are uniformly bounded by $\bar\delta$ as in (\ref{Wbound}), and that the rows of $W$ are independent and with zero mean. However, we will neither assume that a data-driven estimator $\hat D$ is available, nor that specific features of $\t^*$ are known (e.g. $s$ or $|\t^*|_1$). The estimator is defined as a solution of a regularized optimization problem which uses simultaneously $\ell_1$,  $\ell_2$, and $\ell_\infty$ regularization functions.  It can be cast as a convex optimization problem and the solution can be easily computed. We study its rates of convergence in various norms in Section \ref{first}. One of the conclusions is that for $\bar\delta\gg \sqrt{(\log p)/n}$ the new estimator has improved rates of convergence compared to the MU selector. Furthermore, note that the conic estimator $\hat \theta^C$ studied in \cite{BRT2014} can be also applied. Indeed, our setting can be embedded into that of \cite{BRT2014} with $\widehat D$ being the identically zero ${p\times p}$ matrix, which means that we have an estimator of each $\sigma_j^2$ with an error bounded by $\bar\delta^2$. Comparing the bounds yields that the conic estimator $\hat \theta^C$ achieves the same rate as our new estimator if $\bar\delta$ is smaller than or of the order $\big((\log p)/n\big)^{1/4}$. However, there is no bound for $\hat \theta^C$ available when $\bar\delta\gg \big((\log p)/n\big)^{1/4}$.\\

\noindent The second estimator we propose applies to the same setting as in \cite{BRT2014}. The idea of taking advantage of an additional $\ell_\infty$-norm regularization can be used to improve the conic estimator $\hat \theta^C$ of \cite{BRT2014} whenever the rate of convergence of the estimator $\widehat D$  for $\sigma_j^2$, $j=1,\ldots,p$, is slower than $\sqrt{(\log p)/n}$. This motivates us to propose and analyze a modification of the conic estimator. We derive new rates of convergence that can lead to improvements. However, we acknowledge that in the case considered in \cite{BRT2014}, where the rate of convergence of $\widehat D$ is $\sqrt{(\log p)/n}$, there is no gain in the rates of convergence when using the additional $\ell_\infty$-norm regularization.\\

\noindent The paper is organized as follows. Section \ref{sec:assump} contains the notation, main  assumptions and some preliminary lemmas needed to determine threshold constants in the algorithms. The definition and properties of our first estimator are given in Section \ref{first} whereas those of our second procedure can be found in Section \ref{second}. Section 5 contains  simulation results. Some auxiliary lemmas are relegated to an appendix.

%%%%%%%%%%%%%%%%

\section{Notation, assumptions, and preliminary lemmas}\label{sec:assump}

In this section, we introduce the assumptions which will be required to derive the rates of convergence of the proposed estimators. One set of conditions pertains to the design matrix and the second to the errors in the model. We also state preliminary lemmas related to the stochastic error terms. We start by introducing some notation.

\subsection{Notation} Let $J\subset \{1,\ldots,p\}$ be a set of integers. We denote by $|J|$ the cardinality of $J$.  For a vector $\t=(\t_1,\dots,\t_p)$ in $\bR^p$, we denote by $\t_J$ the vector in $\bR^p$ whose $j$th component satisfies $(\t_J)_j = \t_j$ if $j\in J$, and $(\t_J)_j = 0$ otherwise. For $\gamma>0$, the random variable $\eta$ is said to be {\it sub-gaussian with variance parameter} $\gamma^2$ (or shortly {\it $\gamma$-sub-gaussian})
if, for all $t\in\mathbb{R}$,
\begin{equation*}
\E[\text{exp}(t\eta)]\leq
\text{exp}(\gamma^2t^2/2).
\end{equation*}
A random vector $\zeta\in \R^p$ is said to be
{\it sub-gaussian with variance parameter} $\gamma^2$ %(or shortly {\it $\gamma$-sub-gaussian})
if the inner products $(\zeta, v)$ are  $\gamma$-sub-gaussian for any $v\in \R^p$ with $|v|_2=1$.

\subsection{Design matrix}\label{sec:def}

The performance of the estimators that we consider below is influenced by the properties of the Gram
matrix
$$
\Psi = \frac1{n}X^TX.
$$
We will assume that:
{\it
\begin{itemize}
\item[(A1)] The matrix $X$ is deterministic.
\end{itemize}
}
\noindent In order to characterize the behavior of the design matrix, we set \begin{equation*}
m_2=\max_{j=1,\dots,p} \frac1{n}\sum_{i=1}^n X_{ij}^2,
\end{equation*}
where the $X_{ij}$ are the elements of matrix $X$ and we consider the sensitivity characteristics related to the Gram matrix $\Psi$. %introduced in \cite{GT}. %More precisely, we use the closely related variants considered in \cite{BRT2014}.
For $u>0$, define the cone
$$
C_{J}(u)=\big\{\Delta\in\bR^p:\ |\Delta_{J^c}|_1\le
u|\Delta_{J}|_1 \big\},
$$
where $J$ is a subset of $\{1,\hdots,p\}$. For $q\in[1,\infty]$ and
an integer $s\in[1,p]$, the {\em $\ell_q$-sensitivity} (cf. \cite{GT}) is defined as
follows:
$$
\kappa_{q}(s,u)=\min_{J: \ |J|\le s} \Big(\min_{\Delta\in
C_J(u):\ |\Delta|_q=1} \left|\Psi\Delta \right|_{\infty}\Big).
$$
Like in \cite{GT}, we use here the sensitivities to derive the rates of convergence of estimators under sparsity. Importantly, as shown in \cite{GT}, the approach based on sensitivities is more general
than that based on the restricted eigenvalue or the coherence conditions, see also \cite{RT2, GT2, BRT2014}. In particular, under those conditions, we have $\kappa_{q}(s,u) \ge c \,s^{-1/q}$ for some constant $c>0$, which implies the usual optimal bounds for the errors.

\subsection{Disturbances}

Next we turn to the error $W$ in the design and the error $\xi$ in the regression equation.  We will make the following assumptions. \\
{\it
\begin{itemize}
\item[(A2)] The elements of the random vector $\xi$ are independent zero-mean sub-gaussian random variables
with variance parameter~$\sigma^2$.
\item[(A3)] The rows $w_i$, $i=1,\dots,n$, of the noise matrix $W$ are independent zero-mean sub-gaussian random vectors with variance parameter~$\sigma^2_*$. Furthermore, $W$ is independent of $\xi$.
\end{itemize}
}

\subsection{Bounds on the stochastic error terms}

\noindent We now state some useful lemmas from \cite{BRT2014} and \cite{RT2} that provide bounds to various stochastic error terms that play a role in our analysis. We state them here because they introduce the thresholds $\delta_i, \delta_i'$ that will be used in the definition of the estimators. In what follows, $D$ is the diagonal matrix with diagonal elements
$\sigma_j^2$, $j=1,\dots,p$, and for a square matrix $A$, we denote
by $\text{Diag}\{A\}$ the matrix with the same dimensions as $A$,
the same diagonal elements, and all off-diagonal elements
equal to zero.

\begin{lemma}\label{lem2} Let $0<\e<1$ and assume (A1)-(A3). Then, with probability at least $1-\varepsilon$ (for each event),
\begin{eqnarray*}
&&\big|\fracn X^TW\big|_\infty \le \delta_1(\varepsilon),\quad \big| \fracn X^T\xi\big|_\infty \le \delta_2(\varepsilon),\quad \big| \fracn W^T\xi\big|_\infty \le \delta_3(\varepsilon),\\
&& \big|\fracn(W^TW-\text{\rm Diag}\{W^TW\})\big|_\infty \le \delta_4(\varepsilon),\quad\big|\fracn \text{\rm Diag}\{W^TW\}-D\big|_\infty \le \delta_5(\varepsilon),
\end{eqnarray*}
where
\begin{eqnarray*}
&&
\delta_1(\varepsilon)=\sigma_*
\sqrt{\frac{2m_2\log(2p^2/\varepsilon)}{n}}, \quad \delta_2(\varepsilon)=\sigma
\sqrt{\frac{2m_2\log(2p/\varepsilon)}{n}},\\
&& \delta_3(\varepsilon)=\delta_5(\varepsilon)=\bar \delta
(\varepsilon,2p), \quad \delta_4(\varepsilon)=\bar \delta
(\varepsilon,p(p-1)),
\end{eqnarray*}
and for
an integer $N$,
$$
\bar \delta (\varepsilon,N) = \max\left(\gamma_0
\sqrt{\frac{2\log(N/\varepsilon)}{n}},\
\frac{2\log(N/\varepsilon)}{t_0n}\right),
$$
where $\gamma_0, t_0$ are positive constants depending only on $\sigma, \sigma_*$.
\end{lemma}

\begin{lemma} \label{lem3a}
Let $0<\e<1$, $\theta^*\in \R^p$ and assume (A1)-(A3). Then, with probability at least $1-\varepsilon$,
\begin{align*}
&\big|\fracn X^TW\theta^*\big|_{\infty} \leq \delta_1'(\e)|\theta^*|_2,
\end{align*}
where $\delta_1'(\e)= \sigma_{*} \sqrt{\frac{2m_2\log(2p/\e)}{n}}$.
In addition, %$(\log (2p/\e))/n \le 1$ then,
with probability at least $1-\varepsilon$,
\begin{align*}
&\big|\fracn (W^TW-{\rm Diag}\{W^TW\})\theta^*\big|_{\infty}\leq\delta_4'(\e)|\theta^*|_2,
\end{align*}
where %$\delta_1'(\e)= \sigma_{*} \sqrt{\frac{2m_2\log(2p/\e)}{n}}$ and
$$
\delta_4'(\e) =\max\left(\gamma_2
\sqrt{\frac{2\log(2p/\varepsilon)}{n}},\
\frac{2\log(2p/\varepsilon)}{t_2n}\right),
$$
and $\gamma_2, t_2$ are positive constants depending only on $\sigma_*$.
\end{lemma}

\noindent The proofs of Lemmas \ref{lem2} and \ref{lem3a} can be found in \cite{RT2} and \cite{BRT2014} respectively.

%%%%%%%%%%%%%%%%%

\section{$\{\ell_1,\ell_2,\ell_\infty\}$-MU selector}\label{first}

 In this section, we define and analyze our first estimator. It can be seen as a compromise between the MU selector (\ref{Def:MU}) and the conic estimator (\ref{conic}) achieved thanks to an additional $\ell_\infty$-norm regularization. In the setting that we consider now, the estimate $\hat D$ is not available but the rows of the design error matrix $W$ are independent with mean 0, and its entries are uniformly bounded. Formally, in this section, we make the following assumption.
  {\it
\begin{itemize}
 \item[(A4)] Almost surely, $|W|_\infty\leq \bar\delta$.
\end{itemize}
}

\noindent Thus, Assumptions (A1)-(A4) imply the assumptions in \cite{RT1}. However, they neither imply or are implied by the assumptions in \cite{RT2}. That is, it is an intermediary set of conditions relative to the original assumptions for the MU selector in \cite{RT1} and to those for the compensated MU selector in \cite{RT2}.  Importantly, we do not assume that there 
are some accurate estimators of the $\sigma_j^2$. \\

\noindent We consider the estimator $\hat\theta$ such that $(\hat \theta, \hat  t, \hat u) \in \mathbb{R}^{p}\times \R_+ \times \R_+$ is a solution of the following minimization problem
\begin{equation}\label{conicMUnew}
\min_{\theta,t,u}\{|\theta|_1+ \lambda t + \nu u : \,\, \theta\in\Theta, \,
\big|\fracn Z^T(y-Z\theta)\big|_\infty\leq\mu t
+\bar\delta^2 u + \tau, \ |\theta|_2 \le t, |\theta|_\infty \le u\},
\end{equation}
\noindent where $\lambda >0$ and $\nu > 0$ are tuning constants and the minimum is taken over $(\theta,  t,  u) \in \mathbb{R}^{p}\times \R_+ \times \R_+$. This estimator $\hat\theta$ will be further referred to as the $\{\ell_1,\ell_2,\ell_\infty\}$-MU selector. \\

\noindent The estimator above attempts to mimic the conic estimator (\ref{conic}) without an estimator $\widehat D$ for $\sigma_j^2$, $j=1,\ldots,p$. In order to make $\t^*$ feasible for (\ref{conicMUnew}), the contribution of the unknown term $\fracn{\rm Diag}(W^TW)\t^*$ needs to be bounded. This is precisely the role of the extra term $\bar\delta^2 u$ in the constraint since $|\theta|_\infty\leq u$ and $|\fracn{\rm Diag}(W^TW)|_\infty\leq \bar\delta^2$ almost surely. Note that the use of $u$ and $t$ instead of $|\theta|_\infty$ and $|\theta|_2$ in the constraint makes (\ref{conicMUnew}) a convex programming problem.\\

\noindent This new estimator exploits Assumptions (A2)-(A4) to achieve a rate of convergence that is intermediary relative to the rate of the MU selector and to that of the conic estimator. \\

\noindent Set $\mu=\delta'_1(\varepsilon)+\delta'_4(\varepsilon)$ and $\tau=\delta_2(\varepsilon)+\delta_3(\varepsilon)$. Note that $\mu$ and $\tau$  are of order $\sqrt{(\log p)/n}$. The next theorem summarizes the performance of the estimator defined by solving (\ref{conicMUnew}).

\begin{theorem}\label{th:MUnew} Let Assumptions (A1)-(A4) hold. Assume that the true parameter $\t^*$ is $s$-sparse and belongs to $\Th$. Let $0<\e<1$, $1\le q\le \infty$ and $0<\lambda, \nu<\infty$, and let $\hat\theta$ be the $\{\ell_1,\ell_2,\ell_\infty\}$-MU selector. If $\kappa_q(s,1+\lambda+\nu)\geq c s^{-1/q}$  for some constant $c>0$ then, with probability at least $1-7\e$,
\begin{equation}\label{eq:th:MUnew1} |\hat\theta-\theta^*|_q \leq Cs^{1/q} \sqrt{\frac{\log(c'p/\varepsilon)}{n}}(|\theta^*|_1+1) + Cs^{1/q}\bar\delta^2 |\theta^*|_1,
\end{equation}
for some constants $C>0$ and $c'>0$ (here we set $s^{1/\infty}=1$).\\

\noindent If in addition, $\bar\delta^2+ \sqrt{\log(p/\varepsilon)/n} \leq c_1\kappa_1(s,1+\lambda+\nu)$ for some small enough constant~$c_1$ then, with the same probability we have
\begin{equation}\label{eq:th:MUnew2}  |\hat\theta-\theta^*|_q \leq Cs^{1/q} \sqrt{\frac{\log(c'p/\varepsilon)}{n}}(|\theta^*|_2+1) + Cs^{1/q}\bar\delta^2 |\theta^*|_\infty
\end{equation}
for some constants $C>0$ and $c'>0$.\\

\noindent Under the same assumptions with $q=1$, the prediction error admits the following bound,
with the same probability:
\begin{eqnarray}\label{th:MUnew2:pred}
\fracn\big|X(\hat\t-\t^*)\big|_2^2&\le& C s \frac{\log (c'p/\e)}{n} (|\theta^*|_2 +1)^2+Cs\bar\delta^4 |\t^*|_\infty^2\,.
\end{eqnarray}
\end{theorem}

\begin{proof}
We proceed in three steps. Step 1 establishes initial relations and the fact that $\Delta=\hat\theta-\theta^*$ belongs to $C_J(1+\lambda+\nu)$. Step 2 provides a bound on $|\fracn X^TX\Delta|_\infty$. Step 3 establishes the rates of convergence stated in the theorem. We work on the event of probability at least $1-7\varepsilon$ where all the inequalities in Lemmas \ref{lem2} and  \ref{lem3a} are realized. Throughout the proof, $J=\{j : \theta^*_j\ne 0\}.$ We often make use of the inequalities $|\theta|_\infty \le |\theta|_2 \le |\theta|_1, \ \forall \theta\in\R^p.$ \\

\noindent {\it Step 1.} We first note that
\begin{equation}\label{eq:th:MUnew3}  \begin{array}{rl}
|\fracn Z^T(y-Z\theta^*)|_\infty & \leq |\fracn Z^T\xi|_\infty + |\fracn Z^TW\theta^*|_\infty\\
& \leq \delta_2(\varepsilon)+\delta_3(\varepsilon) + |\fracn Z^TW\theta^*|_\infty\end{array}
\end{equation}
with probability at least $1-2\varepsilon$ by Lemma \ref{lem2}. Next,
Lemma \ref{lem3a} and the fact that, due to (\ref{Wbound}), we have  $|\fracn \text{Diag}(W^TW)|_\infty\leq \bar\delta^2$ imply
\begin{equation}\label{eq:th:MUnew4} \begin{array}{rl}
|\fracn Z^TW\theta^*|_\infty & \leq |\fracn X^TW\theta^*|_\infty + |\fracn W^TW\theta^*|_\infty\\
& \leq |\fracn X^TW\theta^*|_\infty + |\fracn (W^TW-\text{Diag}(W^TW))\theta^*|_\infty +|\fracn \text{Diag}(W^TW)\theta^*|_\infty \\
& \leq \delta_1'(\varepsilon)|\theta^*|_2 + \delta_4'(\varepsilon)|\theta^*|_2 + \bar\delta^2|\theta^*|_\infty.
\end{array}
\end{equation}

\noindent Combining \eqref{eq:th:MUnew3} and  \eqref{eq:th:MUnew4} we get that $(\theta,t,u)=(\theta^*,|\theta^*|_2,|\theta^*|_\infty)$ is feasible for the problem (\ref{conicMUnew}), so that
\begin{equation}\label{eq:MU} |\hat \theta|_1 + \lambda |\hat \theta|_2 + \nu |\hat\theta|_\infty
 \le |\hat \theta|_1 + \lambda \hat t + \nu \hat u \le |\theta^*|_1 + \lambda |\theta^*|_2 + \nu |\theta^*
|_\infty.\end{equation}
From \eqref{eq:MU} we easily obtain
 $$|\hat\theta_{J^c}|_1 \leq (1+\lambda+\nu)|\hat\theta_J-\theta^*|_1.$$
 Arguments similar to \eqref{eq:MU} lead to
  $$\hat t - |\theta^*|_2 \leq \frac{|\Delta|_1 + \nu|\Delta|_\infty}{\lambda}\leq \frac{(1+\nu)}{\lambda}|\Delta|_1 \ \ \mbox{and} \ \ \hat u - |\theta^*|_\infty \leq \frac{|\Delta|_1 + \lambda|\Delta|_2}{\nu} \leq \frac{(1+\lambda)}{\nu}|\Delta|_1.$$

\noindent {\it Step 2.}  We have
$$\begin{array}{rl}
|\fracn X^TX\Delta|_{\infty} & \leq |\fracn Z^TX\Delta|_{\infty}+|\fracn W^TX\Delta|_{\infty}\\
& \leq |\fracn Z^TZ\Delta|_{\infty}+|\fracn Z^TW\Delta|_{\infty}+|\fracn W^TX\Delta|_{\infty}\\
& \leq |\fracn Z^T(y-Z\theta^*)|_\infty +|\fracn Z^T(y-Z\hat \theta)|_\infty+|\fracn Z^TW\Delta|_{\infty}+|\fracn W^TX\Delta|_{\infty}.\end{array}$$
The results of Step 1 and of Lemmas \ref{lem2} and \ref{lem3a} imply the following bounds
$$\begin{array}{rl}
 |\fracn Z^T(y-Z\theta^*)|_\infty & \leq  \mu |\theta^*|_2
+\bar\delta^2 |\theta^*|_\infty + \tau,\\
 |\fracn Z^T(y-Z\hat \theta)|_\infty & \leq \mu \hat t
+\bar\delta^2 \hat u + \tau \\
& \leq \mu|\t^*|_2 + \bar\delta^2|\t^*|_\infty + \tau + \{\mu(1+\nu)/\lambda + \bar\delta^2(1+\lambda)/\nu\}|\Delta|_1,\\
|\fracn W^TX\Delta|_{\infty}& \leq \delta_1|\Delta|_1,\\
|\fracn Z^TW\Delta|_{\infty}& \leq |\fracn X^TW\Delta|_{\infty}+|\fracn (W^TW-\text{Diag}(W^TW))\Delta|_{\infty}+|\fracn\text{Diag}(W^TW)\Delta|_{\infty}\\
& \leq \delta_1|\Delta|_1+\delta_4|\Delta|_1+\bar\delta^2|\Delta|_{\infty}.
\end{array}
$$

\noindent These relations and the inequality $|\Delta|_{\infty}\leq |\Delta|_1$ yield that
$$ |\fracn X^TX\Delta|_{\infty} \leq   2\mu |\theta^*|_2
+2\bar\delta^2 |\theta^*|_\infty + 2\tau   + (\bar\delta^2\{(1+\lambda+\nu)/\nu\}+\{(1+\nu)/\lambda\}\mu+2\delta_1+\delta_4)|\Delta|_1.$$

\noindent {\it Step 3.} Next note that $|\Delta|_1 \leq |\hat \theta|_1 + |\theta^*|_1 \leq (2+\lambda+\nu)|\theta^*|_1.$ Letting $$\eta= (\bar\delta^2\{(1+\lambda+\nu)/\nu\}+\{(1+\nu)/\lambda\}\mu+2\delta_1+\delta_4),$$ we have
$$ |\fracn X^TX\Delta|_{\infty} \leq 2\tau + (2\mu +2\bar\delta^2+ (2+\lambda+\nu)\eta)|\theta^*|_1.$$
By the definition of the $\ell_q$-sensitivity,
$$|\fracn X^TX\Delta|_{\infty}\geq \kappa_q(s,1+\lambda+\nu)|\Delta|_{q}.$$
Now, \eqref{eq:th:MUnew1} follows by combining the last two displays and the assumption on $ \kappa_q(s,1+\lambda+\nu)$.
To prove \eqref{eq:th:MUnew2}, we use that
$$\begin{array}{rl}
 |\fracn X^TX\Delta|_{\infty} & \leq   2\mu |\theta^*|_2
+2\bar\delta^2 |\theta^*|_\infty + 2\tau   + \eta|\Delta|_1 \\
&\leq 2\mu |\theta^*|_2
+2\bar\delta^2 |\theta^*|_\infty + 2\tau   + \eta|\fracn X^TX\Delta|_{\infty}/\kappa_1(s,1+\lambda+\nu).\end{array}$$

\noindent Under our conditions, $\eta/\kappa_1(s,1+\lambda+\nu) \leq c'$ for some $0<c'<1$. Thus, we have
\begin{equation}\label{eq:pred1} |\fracn X^TX\Delta|_{\infty} \leq c\big(\mu |\theta^*|_2
+\bar\delta^2 |\theta^*|_\infty + \tau\big),\end{equation}
which implies \eqref{eq:th:MUnew2} in view of the definition of the $\ell_q$-sensitivity and the assumption on $ \kappa_q(s,1+\lambda+\nu)$.\\

\noindent To show \eqref{th:MUnew2:pred}, note first that
\begin{equation}\label{eq:pred}
\fracn \left|X\Delta\right|_2^2 \le \fracn\left|X^TX\Delta\right|_\infty|\Delta|_1.
\end{equation}
By \eqref{eq:th:MUnew2} with $q=1$, 
\begin{equation*}\label{eq:pred2} |\Delta|_1 \leq Cs \sqrt{\frac{\log(c'p/\varepsilon)}{n}}(|\theta^*|_2+1) + Cs\bar\delta^2 |\theta^*|_\infty.
\end{equation*}
Combining this inequality with \eqref{eq:pred1} and \eqref{eq:pred} proves \eqref{th:MUnew2:pred}.
\end{proof}

\medskip

\begin{remark}
We have stated Theorem~\ref{th:MUnew} under Assumption (A4) to make the analysis streamlined with the previous literature, see  \cite{RT1}. However, inspection of the proofs shows that a more general condition can be used. The results of Theorem~\ref{th:MUnew} hold  with probability at least $1-7\varepsilon- \varepsilon'$ if instead  of Assumption (A4) we
require $W$ to satisfy:
$$ |\fracn{\rm Diag}(W^TW)|_\infty \leq \bar\delta^2 $$
with probability at least $1-\varepsilon'$, for some $\varepsilon'>0$.
\end{remark}

\medskip

\noindent Compared to \cite{RT1}, the results in Theorem \ref{th:MUnew} exploit the zero mean condition on the noise matrix~$W$. As in \cite{RT1}, the estimator is consistent as $\bar\delta$ goes to zero. In order to compare the rates in Theorem \ref{th:MUnew} with those for the MU selector, we recall that, by Theorem 3 in \cite{RT1}, the MU selector satisfies  $$|\hat \theta^{MU}-\theta^*|_q \leq Cs^{1/q} \sqrt{\frac{\log(c'p/\varepsilon)}{n}} + Cs^{1/q}(\bar \delta+\bar\delta^2) |\theta^*|_1$$
with probability close to 1.
While both rates share some terms, a term of order $s^{1/q}\bar\delta |\theta^*|_1$ appears only in the rate for the MU selector whereas a term of the order $s^{1/q} \sqrt{\log(c'p/\varepsilon)/n}|\theta^*|_1$ appears only for the $\{\ell_1,\ell_2,\ell_\infty\}$-MU selector. Therefore, the improvement upon the original MU selector is achieved whenever $\bar \delta \gg \sqrt{\log(c'p/\varepsilon)/n}$.\\

\noindent If  the additional condition $\bar\delta^2+ \sqrt{\log(c'p/\varepsilon)/n} \leq c_1\kappa_1(s,1+\lambda+\nu)$ holds, we can use the bound \eqref{eq:th:MUnew2} and a better accuracy is achieved by the proposed estimator. In particular, $|\theta^*|_1$ no longer drives the rate of convergence. The impact of $\bar\delta$ on this rate is  in the term
\begin{equation}\label{14}s^{1/q}\bar\delta^2 |\theta^*|_\infty \ \ \mbox{ instead of } \ \ s^{1/q}(\bar\delta+\bar\delta^2)|\t^*|_1
\end{equation}
for the MU selector. Furthermore, the rate of convergence of the new estimator also has a term of the form $|\theta^*|_2s^{1/q}\sqrt{\log(c'p/\varepsilon)/n}$. Thus the new estimator obtains a better accuracy by exploiting additional assumptions together with the fact that $\bar\delta |\theta^*|_1$ is of larger order than $\sqrt{\log(c'p/\varepsilon)/n}|\theta^*|_2$, which holds whenever $\bar \delta \gg \sqrt{\log(c'p/\varepsilon)/n}$. Finally, the impact of going down from the $\ell_1$-norm to the $\ell_2$- or $\ell_\infty$-norms is not negligible neither. For example, if all non-zero components of $\theta^*$ are equal to the same constant $a>0$, we have $|\t^*|_1=sa$ while $|\t^*|_2= a\sqrt{s}$, and $|\t^*|_\infty= a$. Then, the comparison in \eqref{14} is reduces to comparing
\begin{equation*}\label{*} s^{1/q}\bar\delta^2  \ \ \mbox{ versus } \ \ s^{1+1/q}(\bar\delta+\bar\delta^2),
\end{equation*}
featuring the maximum contrast between the two rates.
\\

\noindent Finally, note that the conic estimator $\hat \theta^C$ studied in \cite{BRT2014} can be also applied under the assumptions of this section. Indeed, our setting can be embedded into that of \cite{BRT2014} with $\widehat D$ being the identically zero ${p\times p}$ matrix, which means that we have an estimator of each $\sigma_j^2$ with an error bounded by $b=\bar\delta^2$. The results in \cite{BRT2014} assume $b=C\sqrt{(\log p)/n}$ but they do not apply to designs with $b$ of larger order. Comparing the bound \eqref{eq:th:MUnew2} in Theorem~\ref{th:MUnew} to the bound \eqref{6} yields that the conic estimator $\hat \theta^C$ achieves the same rate as our new estimator whenever $\bar\delta$  is smaller than or of the order $\big((\log p)/n\big)^{1/4}$. However, there is no bound for $\hat \theta^C$ available when $\bar\delta\gg \big((\log p)/n\big)^{1/4}$.\\

%%%%%%%%%%%%%%%

%%%%%%%%%%%%%%%%%%%%

\section{$\{\ell_1,\ell_2,\ell_\infty\}$-compensated MU selector}\label{second}
\label{sec:conic}

In this section, we discuss a modification of the conic estimator proposed in \cite{BRT2014}. We introduce an additional $\ell_\infty$-norm regularization to better adapt to the estimation error in $\widehat D$. As discussed in the introduction, this is beneficial when the rate of convergence of $\widehat D$ to $D$ is slower than  $\sqrt{(\log p)/n}$, which is not covered by \cite{BRT2014}. Here we consider the same assumptions as in \cite{BRT2014} with the only difference that now we allow for any rate of convergence of $\widehat D$ to $D$. Thus, we replace Assumption (A4) by the following assumption on the availability of estimators for $\sigma_j^2$, $j=1,\ldots,p$.

{\it
\begin{itemize}
\item[(A5)] There exist statistics $\hat{\sigma}_j^2$ and positive numbers $b(\varepsilon)$ such that for any $0<\varepsilon<1$, we have
\begin{equation*}
\mathbb{P}\big[\max_{j=1,\dots,p}|\hat{\sigma}_j^2-\sigma_j^2|\geq
b(\e)\big]\leq \varepsilon.
\end{equation*}
\end{itemize}
}

\noindent In what follows, we fix $\e$ and set $$\mu=\delta_1'(\varepsilon)+\delta_4'(\varepsilon),~~\tau=\delta_2(\varepsilon)+\delta_3(\varepsilon) \ \ \mbox{and} \ \ \b = b(\e)+\delta_5(\varepsilon).$$ We are particularly interested in cases where $\b$ is of larger order than $\sqrt{(\log p)/n}$. To define the estimator, we consider the following minimization problem:

\begin{equation}\label{conicNew}
\min_{\theta,t,u}\{|\theta|_1+ \lambda t + \nu u : \,\, \theta\in\Theta, \,
\big|\fracn Z^T(y-Z\theta)+\widehat{D}\theta\big|_\infty\leq\mu t
+\b u+\tau, \ |\theta|_2 \le t, \ |\theta|_\infty \le u\}.
\end{equation}
Here, $\lambda >0$ and $\nu > 0$ are tuning constants and the minimum is taken over $(\theta,  t,  u) \in \mathbb{R}^{p}\times \R_+ \times \R_+$.

\medskip

\noindent  Let $(\hat\theta, \hat t, \hat u)$ be a solution of (\ref{conicNew}). We take $\hat\theta$ as estimator of $\theta^*$ and we  call it the $\{\ell_1,\ell_2,\ell_\infty\}$-compensated MU selector. The rates of convergence of this estimator are given in the next theorem.

\begin{theorem}\label{th:conic}
Let Assumptions (A1)-(A3), and (A5) hold. Assume that the true parameter $\t^*$ is $s$-sparse and belongs to $\Th$. Let $0<\e<1$ and $1\le q\le \infty$.
Suppose also that
\begin{equation}\label{ass_kappa}\kappa_q(s,1+\lambda+\nu)\ge cs^{-1/q}
\end{equation}
 for some constant $c>0$ and that
 \begin{equation}\label{hyps}
 s\le c_1\min\{\sqrt{n/\log (p/\e)}, 1/b(\e)\},\end{equation} for some small enough constant $c_1>0$. Let $\hat\theta$ be the $\{\ell_1,\ell_2,\ell_\infty\}$-compensated MU selector. Then, with probability at least $1-8\e$,
\begin{eqnarray}\label{11a}
|\hat\t-\t^*|_q&\le& C s^{1/q} \sqrt{\frac{\log (c'p/\e)}{n}} (|\theta^*|_2 +1)+Cs^{1/q}b(\varepsilon)|\t^*|_\infty,
\end{eqnarray}
for some constants $C>0$ and $c'>0$ (here we set $s^{1/\infty}=1$).\\

\noindent Under the same assumptions with $q=1$, the prediction error admits the following bound,
with the same probability:
\begin{eqnarray}\label{11c}
\fracn\big|X(\hat\t-\t^*)\big|_2^2&\le& C s \frac{\log (c'p/\e)}{n} (|\theta^*|_2 +1)^2+Csb^2(\e)|\t^*|_\infty^2\,.
\end{eqnarray}
\end{theorem}
\begin{proof}
Throughout the proof, we assume that we are on the event of probability
at least $1-8\varepsilon$ where the results of Lemmas \ref{lem3}, \ref{lem4} and \ref{lem5} in the Appendix hold. Property (\ref{c1}) in Lemma \ref{lem4} implies that $\Delta=\hat\t - \t^*$ is in the cone $C_J(1+\lambda+\nu)$, where $J=\{j : \theta^*_j\ne 0\}.$ Therefore, by the definition of the $\ell_q$-sensitivity and Lemma~\ref{lem5}, we have
$$
\kappa_q(s,1+\lambda+\nu)|\Delta|_q\le \big|\fracn X^TX\Delta\big|_\infty \le \mu_0+ \mu_1 |\hat\theta-\theta^*|_1 +  \mu_2 |\t^*|_2+ \mu_\infty|\t^*|_\infty,
$$
where $\mu_0$ and $\mu_2$ are of order $\sqrt{\fracn\log (c'p/\e)}$, and $\mu_1$ and $\mu_\infty$ are of order $\sqrt{\fracn\log (c'p/\e)}+b(\e)$. Using again (\ref{c1}), we have
\begin{eqnarray*}
|\Delta|_1 &=& |\Delta_{J^c}|_1+|\Delta_J|_1\le (2+\lambda+\nu)|\Delta_J|_1\\
&\le & (2+\lambda+\nu)s^{1-1/q}|\Delta_{J}|_q \le (2+\lambda+\nu) s^{1-1/q}|\Delta|_q.
\end{eqnarray*}
It follows that
$$
(\kappa_q(s,1+\lambda+\nu)- (2+\lambda+\nu)\mu_1 s^{1-1/q})|\Delta|_q \le  \mu_0+\mu_2 |\theta^*|_2 + \mu_\infty |\t^*|_\infty,
$$
which implies, by \eqref{ass_kappa},
$$
(c- (2+\lambda+\nu)\mu_1 s) s^{-1/q}|\Delta|_q \le   \mu_0+\mu_2 |\theta^*|_2 + \mu_\infty |\t^*|_\infty,
$$
in view of the assumptions of the theorem. Recall that $\mu_1 \le a \{\sqrt{\log (c'p/\e)/n}+b(\e)\}$, where $a>0$ is a constant. Therefore, since we assume that $s \leq c_1 \min\{\sqrt{n/\log (p/\e)}, 1/b(\e)\}$,  (\ref{11a}) follows if $c_1$ is small enough.\\

\noindent To prove (\ref{11c}), we use \eqref{eq:pred}. 
Remark that from (\ref{11a}) with $q=1$, we have
$$|\Delta|_1\leq C s \sqrt{\frac{\log (c'p/\e)}{n}} (|\theta^*|_2 +1) + Csb(\varepsilon)|\t^*|_\infty.$$
Lemma \ref{lem5} in the Appendix yields
\begin{equation}\label{sb}
\big|\fracn X^TX\Delta\big|_\infty \le \mu_0+ \mu_1 |\hat\theta-\theta^*|_1 +  \mu_2 |\t^*|_2+ \mu_\infty|\t^*|_\infty.
\end{equation}
Combining the above bound for $|\Delta|_1$ and (\ref{sb}), we get
$$
\fracn \left|X\Delta\right|_2^2 \le C\frac{s\log (c'p/\e)}{n} (|\theta^*|_2 +1)^2 + Csb^2(\varepsilon)|\t^*|_\infty^2
$$
since $\mu_1 s \leq C''$ for some constant $C''>0$ under our assumptions. This proves (\ref{11c}).
\end{proof}

\bigskip

\noindent Theorem~\ref{th:conic} generalizes the results in \cite{BRT2014} to estimators $\widehat D$ that converge with rate $b(\e)$ of larger order than $\sqrt{(\log p)/n}$. At the same time, if $b(\e)$ is smaller than $\sqrt{(\log p)/n}$, both the conic estimator $\hat \theta^C$ of \cite{BRT2014} and the $\{\ell_1,\ell_2,\ell_\infty\}$-compensated MU selector achieve the same rate of convergence.\\

\noindent For such designs that condition (\ref{hyps}) does not hold, the conclusions of Theorem \ref{th:conic} need to be slightly modified  as shown in the next theorem. \\ %This suggests a robust performance of the estimator in extreme cases.

\begin{theorem}\label{th:conicRelaxed}
Let Assumptions (A1)-(A3), and (A5) hold. Assume that the true parameter $\t^*$ is $s$-sparse and belongs to $\Th$.
Let $0<\e<1$ and $1\le q\le \infty$. Let $\hat\theta$ be the $\{\ell_1,\ell_2,\ell_\infty\}$-compensated MU selector. Then, with probability at least $1-8\e$,
\begin{eqnarray}\label{eq:th:conicRelaxed:1}
|\hat\t-\t^*|_q&\le& \frac{C}{\kappa_q(s,1+\lambda+\nu)} \left\{\sqrt{\frac{\log (c'p/\e)}{n}} (|\theta^*|_1 +1)+b(\varepsilon)|\theta^*|_1\right\},
\end{eqnarray}
for some constants $C>0$ and $c'>0$, and
the prediction error admits the following bound,
with the same probability:
\begin{eqnarray}\label{eq:th:conicRelaxed:2}
\fracn\big|X(\hat\t-\t^*)\big|_2^2\le C \min \left\{ \frac{\frac{\log (c'p/\e)}{n} (|\theta^*|_1 +1)^2+b^2(\e)|\t^*|_1^2}{\kappa_1(s,1+\lambda+\nu)}  , \, \sqrt{\frac{\log (c'p/\e)}{n}} (|\theta^*|_1 +1)^2+b(\e)|\t^*|_1^2\right\}.
\end{eqnarray}
\end{theorem}

\begin{proof}
Again, throughout the proof, we assume that we are on the event of probability at least $1-8\varepsilon$ where the results of Lemmas \ref{lem3}, \ref{lem4} and \ref{lem5} in the Appendix hold. Property (\ref{c1}) in Lemma \ref{lem4} implies that $\Delta=\hat\t - \t^*$ is in the cone $C_J(1+\lambda+\nu)$, where $J=\{j : \theta^*_j\ne 0\}.$ Since 
\begin{eqnarray}\label{eq:th:conicRelaxed:3}|\Delta|_1\leq |\hat \theta|_1 + |\t^*|_1 \leq \{ |\t^*|_1 +\lambda|\t^*|_2+\nu|\t^*|_\infty\} + |\theta^*|_1 \leq (2+\lambda+\nu)|\theta^*|_1,
\end{eqnarray}  we obtain
$$
 \big|\fracn X^TX\Delta\big|_\infty \le \mu_0 + \mu_1 |\Delta|_1 +  \mu_2 |\t^*|_2+ \mu_\infty|\t^*|_\infty \leq \mu_0+(\mu_1+\mu_2+\mu_\infty)(2+\lambda+\nu)|\theta^*|_1.
$$
Therefore
$$
\kappa_q(s,1+\lambda+\nu)|\Delta|_q \le \mu_0+(\mu_1+\mu_2+\mu_\infty)(2+\lambda+\nu)|\theta^*|_1,
$$
which implies \eqref{eq:th:conicRelaxed:1}. Note also that, due to \eqref{eq:pred},
the above displays immediately imply the bound on the prediction risk given by the second term under the minimum in  \eqref{eq:th:conicRelaxed:2}. The first term under the minimum in  \eqref{eq:th:conicRelaxed:2} is obtained by combining \eqref{eq:pred}, \eqref{eq:th:conicRelaxed:1} with $q=1$, and \eqref{eq:th:conicRelaxed:3}.
\end{proof}

\section{Simulations}\label{sec:simulations}

This section aims to illustrate the finite sample performance of the proposed estimators. We will focus on the $\{\ell_1,\ell_2,\ell_\infty\}$-compensated MU selector only. We consider the following data generating process
$$ y_i = x_i^T\theta^* + \xi_i, \ \ \ \ z_i = x_i + w_i.$$
Here, $\xi_i, w_i, x_i$ are independent and $\xi_i \sim {\cal N}(0,\sigma^2)$, $w_{i}\sim {\cal N}(0,\sigma_*^2I_{p\times p})$, $x_i\sim {\cal N}(0,\Sigma)$ where $I_{p\times p}$ is the identity matrix and $\Sigma$ is $p\times p$ matrix with elements $\Sigma_{ij}=\rho^{|i-j|}$.
We consider the vector of unknown parameters $\theta^*= 1.25(1,1,1,1,1,0,\ldots,0)^T $. We set $\sigma = 0.128$, $\sigma_*^2=0.5$, and $\rho=0.25$. We assume that $\sigma$ is known and we set $\hat D = D = \sigma_*^2I_{p\times p}$. The penalty parameters are set as $\tau = \sigma\sqrt{\log(p/\varepsilon)/n}$, $b(\varepsilon)=\sigma_*^2\sqrt{\log(p/\varepsilon)/n}$, for $\varepsilon = 0.05$.\\

\noindent In our first set of simulations, we illustrate the finite sample performance of the proposed estimator by setting $\lambda = \nu \in \{0.25, 0.5, 0.75,  1\}$. The $\{\ell_1,\ell_2,\ell_\infty\}$-compensated MU selector will be denoted by $\{\ell_1,\ell_2,\ell_\infty\}$. We compare its performance with other recent proposals in the literature, namely the conic estimator (denoted as Conic ($\lambda$) for $\lambda = 0.25, 0.5, 0.75, 1$), and the Compensated MU selector (cMU). We also provide the (infeasible) Dantzig selector which knows $X$ (Dantzig X) and the Dantzig selector that uses only $Z$ (Dantzig Z) as additional benchmark for the performance. 

\begin{table}[h!]
\begin{center}
\begin{tabular}{l|ccc|ccc}

 &  \multicolumn{3}{|c|}{$n=300$ and $p=10$} & \multicolumn{3}{|c}{$n=300$ and $p=50$}\\

   Method ($\lambda=\nu$)  & Bias  & RMSE  &  PR     & Bias & RMSE  &  PR\\
   \hline
  Dantzig X & 0.0265486    &  0.0321528      & 0.0349530 & 0.0301636    &  0.0349420      & 0.0386731 \\
  Dantzig Z & 0.5892699    &  0.6218173      & 0.7118256 & 0.6032541    &  0.7246990      & 0.7526539 \\
  cMU  & 0.6002801  &  0.6526144      & 0.7375240 & 0.6684987  &  0.7074681      & 0.8148175 \\
  Conic (0.25) & 1.9261733   & 1.9567318 &     2.3165088  & 1.9952936   & 2.0190105 &     2.4085353 \\
  $\{\ell_1,\ell_2,\ell_\infty\}$(0.25)  & 1.7922416    &  1.8349666      & 2.1453927  & 1.9035308    &  1.9325326      & 2.2875796 \\
  Conic (0.5)  & 0.3184083   & 0.4161670 &     0.4326569   & 0.3668194   & 0.4395404 &     0.4781078 \\
  $\{\ell_1,\ell_2,\ell_\infty\}$ (0.5)  & 0.2137347    &  0.3505829      & 0.3382480 & 0.3489980    &  0.4491837      & 0.4605638 \\
  Conic (0.75)  & 0.3179691   & 0.4158134 &     0.4322128 & 0.3668194   & 0.4395404 &     0.4781078 \\
  $\{\ell_1,\ell_2,\ell_\infty\}$ (0.75)  & 0.2085334    &  0.3459298      & 0.3330411 & 0.2699453    &  0.3786945      & 0.3896168 \\
  Conic (1) & 0.3179691   & 0.4158134 &     0.4322128 & 0.3661721   & 0.4390614 &     0.4773173 \\
  $\{\ell_1,\ell_2,\ell_\infty\}$ (1)  & 0.2078373    &  0.3455287      & 0.3324356 & 0.2483137    &  0.3691060      & 0.3736929 \\
\hline
  \end{tabular}
  \caption{\footnotesize Simulation results for 100 replications. For each estimator we provide average bias (Bias), average root-mean squared error (RMSE), and average prediction risk (PR).}\label{Table:MC1}
\end{center}
\end{table}

\begin{table}[h!]
\begin{center}
\begin{tabular}{l|ccc|ccc}

    &  \multicolumn{3}{|c|}{$n=300$ and $p=100$} & \multicolumn{3}{|c}{$n=300$ and $p=300$}\\
   Method ($\lambda=\nu$)  & Bias  & RMSE  &  PR     & Bias & RMSE  &  PR\\
\hline
  Dantzig X & 0.0317776    &  0.0366155      & 0.0403419 & 0.0344617    &  0.0387848      & 0.0436396 \\
  Dantzig Z & 0.6039890    &  0.8364059      & 0.7910512 & 0.6334052    &  1.0775665      & 0.8824695 \\
  cMU  & 0.6908240  &  0.7359536      & 0.8472447 & 0.7228791  &  0.7653174      & 0.8843476 \\
  Conic (0.25)  & 2.0196204   & 2.0428152 &     2.4429977 & 2.0833543   & 2.0985979 &     2.5281871 \\
  $\{\ell_1,\ell_2,\ell_\infty\}$(0.25)  & 1.9363225    &  1.9646153      & 2.3321286 & 2.0016163    &  2.0247679      & 2.4181903 \\
  Conic (0.5)  & 0.5032353   & 0.6479385 &     0.6390150  & 0.6809176   & 0.8886359 &     0.8367831 \\
  $\{\ell_1,\ell_2,\ell_\infty\}$ (0.5)  & 0.4170439    &  0.5207218      & 0.5436218 & 0.4694103    &  0.5507253      & 0.5975351 \\
  Conic (0.75)  & 0.3849631   & 0.4699933 &     0.5082582 & 0.4195124   & 0.4964321 &     0.5428568 \\
  $\{\ell_1,\ell_2,\ell_\infty\}$ (0.75) & 0.3250997    &  0.4312186      & 0.4512656 & 0.3869566    &  0.4747343      & 0.5104562 \\
  Conic(1)  & 0.3811186   & 0.4673239 &     0.5043246  & 0.4047078   & 0.4846393 &     0.5271225 \\
  $\{\ell_1,\ell_2,\ell_\infty\}$ (1)  & 0.2907918    &  0.4155573      & 0.4242000 & 0.3573025    &  0.4569624      & 0.4819208 \\
   \hline
  \end{tabular}
  \caption{\footnotesize Simulation results for 100 replications. For each estimator we provide average bias (Bias), average root-mean squared error (RMSE), and average prediction risk (PR).}\label{Table:MC2}
\end{center}
\end{table}

\noindent Tables \ref{Table:MC1} and \ref{Table:MC2} provide the performance of the proposed estimator when $\lambda=\nu$ and the performance of various benchmarks. As discussed in the literature, ignoring the error-in-variables issue can lead to worse performance as seen from the performance of Dantzig Z compared to the (infeasible) Dantzig X. The conic estimator performs better than the compensated MU selector (cMU) when $\lambda \in \{0.5, 0.75, 1\}$. The comparison of the proposed estimator and the conic estimator is easier to establish as we can parametrize them by $\lambda$ (as we set $\lambda = \nu$). In this case the conic estimator penalizes more aggressively the uncertainty of not knowing $\sigma_j^2$. In essentially all cases\footnote{The conic estimator performs slightly better only with respect to RMSE in the case of $\lambda =0.5$. For all other parameters and metrics, the proposed estimator performs slightly better or substantially better.} the proposed estimator yields improvements. The introduction of $\ell_\infty$-norm regularization seems to alleviate regularization bias. Nonetheless, when setting $\lambda =0.25$  both the conic estimator and the proposed estimator fail in the experiment. This failure occurs by not having enough penalty to control $t - |\theta|_2$ and $u - |\theta|_\infty$ which leads to a large right hand side $\mu t+\b u+\tau$ in  the  constraint
$$ \big|\fracn Z^T(y-Z\theta)+\widehat{D}\theta\big|_\infty\leq\mu t
+\b u+\tau$$
in (\ref{conicNew}) and similarly the right hand side  $\mu t+\tau$ in (\ref{conic}). In turn, this leads to substantial regularization bias and therefore underfitting. In fact, detailed inspection of estimators in that case reveals that coefficients are very close to zero for both the conic and the proposed estimator.\\

\noindent In the second set of simulations, we explore the performance of the proposed estimator for the case $\lambda\neq \nu$.
Moreover, we also study a modified estimator that contains safeguard constraints. These constraints aim to mitigate the problem discussed above. The safeguard constraints are described in Remark \ref{Rem:Safeguard} below. We denote by $\{\ell_1,\ell_2,\ell_\infty\}^*$ the estimator computed with the safeguards.

\medskip

\begin{remark}[Safeguard Constraints]\label{Rem:Safeguard}
In order to further bound $t$ and $u$, we can add constraints that exploit that $|\cdot |_q \leq |\cdot|_1$ for $q\geq 1$. Therefore, the constraints
$$ \theta = \theta^+ - \theta^-, \ \ \theta^+\geq 0, \ \ \theta^-\geq 0, \ \ w = \sum_{j=1}^p\{\theta^++\theta^-\}, \ \  t \leq w, \ \ \mbox{and} \ \  u \leq w$$
preserve the convexity of the optimization problem and can potentially yield additional performance.\end{remark}

\medskip

\noindent
We consider the same design as before and we explore some combinations of values $$(\lambda,\nu)\in \{0.25,0.5,0.75,1\}\times\{0.25,0.5,0.75,1\}$$ for both proposed estimators (with and without the safeguard constraints).\\

\begin{table}[h!]
\begin{center}
\begin{tabular}{l|ccc|ccc}

  &  \multicolumn{3}{|c|}{$n=300$ and $p=10$} & \multicolumn{3}{|c}{$n=300$ and $p=50$}\\

   Method $(\lambda,\nu)$   & Bias  & RMSE  &  PR     & Bias & RMSE  &  PR\\
   \hline
  $\{\ell_1,\ell_2,\ell_\infty\}$ \ (1,1)  & 0.2078373    &  0.3455287      & 0.3324356 & 0.2483137    &  0.3691060      & 0.3736929 \\
   $\{\ell_1,\ell_2,\ell_\infty\}^*$  (1,1)   & 0.2078373    &  0.3455287      & 0.3324356  & 0.2483137    &  0.3691060      & 0.3736929\\
  $\{\ell_1,\ell_2,\ell_\infty\}$ \ (1,0.5)  & 0.2534465    &  0.3997941      & 0.3725479  & 0.5214272    &  0.7086267      & 0.6514348 \\
  $\{\ell_1,\ell_2,\ell_\infty\}^*$  (1,0.5)  & 0.2392491    &  0.3623416      & 0.3569492 & 0.3980543    &  0.4729990      & 0.5112310 \\
   $\{\ell_1,\ell_2,\ell_\infty\}$ \ (0.5,1) & 0.2077228    &  0.3455690      & 0.3322088 & 0.2448911    &  0.3690180      & 0.3723095 \\
  $\{\ell_1,\ell_2,\ell_\infty\}^*$  (0.5,1)   & 0.2077228    &  0.3455690      & 0.3322088 & 0.2448911    &  0.3690180      & 0.3723095 \\
$\{\ell_1,\ell_2,\ell_\infty\}$ \ (0.75,0.75)  & 0.2085334    &  0.3459298      & 0.3330411 & 0.2699453    &  0.3786945      & 0.3896168 \\
$\{\ell_1,\ell_2,\ell_\infty\}^*$ (0.75,0.75)  & 0.2085334    &  0.3459297      & 0.3330411 & 0.2699453    &  0.3786945      & 0.3896168 \\
$\{\ell_1,\ell_2,\ell_\infty\}$ \ (0.25,1)    & 0.2078663    &  0.3458796      & 0.3322444 & 0.2439496    &  0.3684173      & 0.3715836 \\
  $\{\ell_1,\ell_2,\ell_\infty\}^*$  (0.25,1)  & 0.2078663    &  0.3458796      & 0.3322444 & 0.2439496    &  0.3684173      & 0.3715836 \\
  $\{\ell_1,\ell_2,\ell_\infty\}$ \ (0.5,0.5)  & 0.2137347    &  0.3505829      & 0.3382480 & 0.3489980    &  0.4491837      & 0.4605638 \\
  $\{\ell_1,\ell_2,\ell_\infty\}^*$  (0.5,0.5)   & 0.2137347    &  0.3505827      & 0.3382479 & 0.3382218    &  0.4225007      & 0.4490958 \\
  $\{\ell_1,\ell_2,\ell_\infty\}$ \ (0.25,0.5)    & 0.2114159    &  0.3502938      & 0.3369809 & 0.3188151    &  0.4086438      & 0.4313163\\
  $\{\ell_1,\ell_2,\ell_\infty\}^*$  (0.25,0.5) & 0.2114159    &  0.3502938      & 0.3369809 & 0.3188151    &  0.4086438      & 0.4313163 \\

 $\{\ell_1,\ell_2,\ell_\infty\}$ \ (0.25,0.25)  & 1.7922416    &  1.8349666      & 2.1453927  & 1.9035308    &  1.9325326      & 2.2875796 \\
 $\{\ell_1,\ell_2,\ell_\infty\}^*$ (0.25,0.25)   & 0.5477221    &  0.6050091      & 0.6780050 & 0.6151622    &  0.6574460      & 0.7535672  \\
   \hline
  \end{tabular}
  \caption{\footnotesize Simulation results for 100 replications. For each estimator we provide average bias (Bias), average root-mean squared error (RMSE), and average prediction risk (PR).}\label{Table:MC3}
\end{center}
\end{table}

\begin{table}[h!]
\begin{center}
\begin{tabular}{l|ccc|ccc}

  &  \multicolumn{3}{|c|}{$n=300$ and $p=100$} & \multicolumn{3}{|c}{$n=300$ and $p=300$}\\

   Method $(\lambda,\nu)$   & Bias  & RMSE  &  PR     & Bias & RMSE  &  PR\\
   \hline
  $\{\ell_1,\ell_2,\ell_\infty\}$ \ (1,1)  & 0.2907918    &  0.4155573      & 0.4242000 & 0.3573025    &  0.4569624      & 0.4819208 \\
  $\{\ell_1,\ell_2,\ell_\infty\}^*$  (1,1)   & 0.2907918    &  0.4155573      & 0.4242000 & 0.3573084    &  0.4569653      & 0.4819268\\
  $\{\ell_1,\ell_2,\ell_\infty\}$ \ (1,0.5)   & 0.6707248 &  0.8687948& 0.8260765 & 1.0995021&  1.2843061 & 1.3224733 \\
  $\{\ell_1,\ell_2,\ell_\infty\}^*$  (1,0.5)  & 0.4713115 &  0.5469680 & 0.5998890 & 0.5813572 &  0.6440090 & 0.7214057 \\
  $\{\ell_1,\ell_2,\ell_\infty\}$ \ (0.5,1)   & 0.2813842    &  0.4123716      & 0.4183949  & 0.3434854    &  0.4501870      & 0.4705890 \\
  $\{\ell_1,\ell_2,\ell_\infty\}^*$ (0.5,1)   & 0.2813842    &  0.4123716      & 0.4183949  & 0.3434113    &  0.4501763      & 0.4705578 \\
  $\{\ell_1,\ell_2,\ell_\infty\}$ \ (0.75,0.75) & 0.3250997    &  0.4312186      & 0.4512656 & 0.3869566    &  0.4747343      & 0.5104562 \\
$\{\ell_1,\ell_2,\ell_\infty\}^*$ (0.75,0.75)   & 0.3250997    &  0.4312186      & 0.4512656 & 0.3869382    &  0.4747230      & 0.5104387 \\
  $\{\ell_1,\ell_2,\ell_\infty\}$ \ (0.25,1) & 0.2791982    &  0.4113275      & 0.4166136  & 0.3392525    &  0.4482318      & 0.4674070 \\
  $\{\ell_1,\ell_2,\ell_\infty\}^*$ (0.25,1)  & 0.2790174    &  0.4111578      & 0.4163974 & 0.3386830    &  0.4478506      & 0.4667958  \\
  $\{\ell_1,\ell_2,\ell_\infty\}$ \ (0.5,0.5)  & 0.4170439    &  0.5207218      & 0.5436218 & 0.4694103    &  0.5507253      & 0.5975351 \\
  $\{\ell_1,\ell_2,\ell_\infty\}^*$  (0.5,0.5) & 0.3977208    &  0.4819926      & 0.5218939 & 0.4590288    &  0.5311414      & 0.5863702 \\
  $\{\ell_1,\ell_2,\ell_\infty\}$ \ (0.25,0.5)  & 0.3726889    &  0.4645454      & 0.4983916 & 0.4324230    &  0.5082528      & 0.5569916 \\
  $\{\ell_1,\ell_2,\ell_\infty\}^*$  (0.25,0.5)  & 0.3718829    &  0.4639374      & 0.497499  & 0.4357035    &  0.5115011      & 0.5608083 \\
  $\{\ell_1,\ell_2,\ell_\infty\}$ \ (0.25,0.25)  & 1.9363225    &  1.9646153      & 2.3321286 & 2.0016163    &  2.0247679      & 2.4181903 \\
   $\{\ell_1,\ell_2,\ell_\infty\}^*$ (0.25,0.25)   & 0.6365329    &  0.6854681      & 0.7851421  & 0.6669880    &  0.7127657      & 0.8198798 \\
   \hline
  \end{tabular}
  \caption{\footnotesize Simulation results for 100 replications. For each estimator we provide average bias (Bias), average root-mean squared error (RMSE), and average prediction risk (PR).}\label{Table:MC4}
\end{center}
\end{table}

\noindent Tables \ref{Table:MC3} and \ref{Table:MC4} show the performance for different values of $\lambda$ and $\nu$. We note that these parameters seem to have different impact on the finite sample performance even if $\lambda+\nu$ is kept constant. Importantly, we observe that the addition of safeguard constraints virtually always leads to improvements although small (even zero sometimes) for most of the tested parameter values. In the case $\lambda < \nu$, using safeguard constraints makes almost no difference and overall performance of both estimators is better. In contrast, the estimators perform worse when $\lambda > \nu$ and the safeguard constraints lead to improvements. Finally, as expected, the safeguard constraints improve substantially the performance when $\lambda = \nu = 0.25$. In that case, the performance becomes comparable to that of the cMU estimator. Essentially, the safeguard constraints help to avoid severe underfitting. They are very helpful when the performance is below of what can be achieved. Nonetheless, we recommend to keep them in all cases as it does not impact negatively the estimator and the additional computational burden seems minimal.

%\section{Conclusion}\label{sec:conclu}
%
%\noindent In conclusion,....\\
%

\section*{Appendix: Auxiliary lemmas}\label{appendixA}

In what follows, we write for brevity $\d_i=\d_i(\e)$, $\d_i'=\d_i'(\e)$, and we set $\Delta=\hat\t - \t^*$, $J=\{j : \theta^*_j\ne 0\}.$

\begin{lemma}\label{lem3}
Assume (A1)-(A3) and (A5). Then with probability at least $1-6\varepsilon$, the pair $(\theta, t, u)=(\t^*, |\t^*|_2,|\t^*|_\infty) $ belongs to the feasible set of the minimization problem
(\ref{conicNew}).
\end{lemma}
\begin{proof} First, note that
$Z^T(y-Z\theta^*)+n{\widehat D}\theta^*$ is equal to
\begin{align*}
&-X^TW\theta^*+X^T\xi+W^T\xi-(W^TW-\text{Diag}\{W^TW\})\theta^*\\
&-(\text{Diag}\{W^TW\}-nD)\theta^*+n(\widehat D-D)\theta^*.
\end{align*}
By definition of $\delta_i$ and $b$, with probability at least
$1-4\varepsilon$, we have
\begin{align}
&|\fracn X^T\xi|_{\infty}+|\fracn W^T\xi|_{\infty}\leq \delta_2+\delta_3\label{lem3.2}\\
&|(\fracn \text{Diag}\{W^TW\}-D)\theta^*|_{\infty}\leq|\fracn \text{Diag}\{W^TW\}-D|_{\infty}|\theta^*|_\infty\leq \delta_5|\theta^*|_\infty\label{lem3.3}\\
&|(\widehat D-D)\theta^*|_{\infty}\leq b(\e)|\theta^*|_\infty,\label{lem3.4}
\end{align}
where in (\ref{lem3.3}) and (\ref{lem3.4}) we have used that the considered matrices are diagonal. Also, by Lemma \ref{lem3a}, with probability at least
$1-2\varepsilon$, we have
\begin{align}
&|\fracn X^TW\theta^*|_{\infty}\leq \delta_1'|\theta^*|_2\label{ineg}\\
&|\fracn(W^TW-\text{Diag}\{W^TW\})\theta^*|_{\infty}\leq\delta_4'|\theta^*|_2.\label{ineg1}
\end{align}
Combining the decomposition of $Z^T(y-Z\theta^*)+n{\widehat D}\theta^*$ together with (\ref{lem3.2})-(\ref{ineg1}), we find that $$\big|\fracn Z^T(y-Z\theta^*)+\widehat{D}\theta^*\big|_\infty\leq\mu |\theta^*|_2
+b|\t^*|_\infty+\tau,$$ with probability at
least $1-6\varepsilon$, which implies the lemma.
\end{proof}
$~$\\

\begin{lemma}\label{lem4} Let $\hat \theta$ be the $\{\ell_1,\ell_2,\ell_\infty\}$-compensated MU-selector. Assume (A1)-(A3) and (A5). Then with probability at least $1-6\varepsilon$ (on the same event as in Lemma \ref{lem3}), we have
\begin{align}
& |(\hat\theta-\theta^*)_{J^c}|_1\leq (1+\lambda+\nu)|(\hat \theta-\theta^*)_{J}|_1, \label{c1}\\
&  \hat t - |\t^*|_2 \le \{(1+\nu)/\lambda\} |\hat \theta-\t^*|_1 \ \ \mbox{and} \ \ \hat u -|\t^*|_\infty \le \{(1+\lambda)/\nu\}|\hat \theta-\t^*|_1.\label{c2}
\end{align}
\end{lemma}

\begin{proof} Set $\Delta=\hat\theta-\theta^*$. On the event of Lemma \ref{lem3}, $(\t^*, |\t^*|_2,|\t^*|_\infty)$ belongs to the feasible set of the minimization problem (\ref{conic}). Consequently,
\begin{align}
& |\hat\theta|_1+ \lambda|\hat\theta|_2+\nu|\hat\theta|_\infty\leq |\hat\theta|_1 + \lambda \hat t +\nu \hat u \le |\t^*|_1 + \lambda |\t^*|_2+\nu|\t^*|_\infty. \label{c3}
\end{align}
This implies
%$$|\hat\theta|_1\le |\theta^*|_1 + \lambda |\Delta|_2\le |\theta^*|_1 + \lambda |\Delta|_1,$$
%and so
%$$|\hat\theta|_1-|\hat\theta_J|_1\le |\theta^*|_1-|\hat\theta_J|_1 + \lambda |\Delta|_1.$$
%Hence we easily deduce
%$$|\Delta_{J^c}|_1\le |\Delta_J|_1 + \lambda |\Delta|_1.$$
%Since $\lambda=1/2$, (\ref{c1}) follows.
$$
|\Delta_{J^c}|_1+\lambda |\Delta_{J^c}|_2+\nu|\Delta_{J^c}|_\infty \le |\Delta_J|_1 + \lambda |\Delta_J|_2+\nu|\Delta_{J}|_\infty\le (1+\lambda+\nu)|\Delta_J|_1,
$$
and so
$$|\Delta_{J^c}|_1 \le (1+ \lambda+\nu) |\Delta_J|_1.$$
and (\ref{c1}) follows. To prove (\ref{c2}), it suffices to note that (\ref{c3}) implies
\begin{align*}
& \lambda \hat t \le |\theta^*|_1 -|\hat\theta|_1  + \lambda |\theta^*|_2 +\nu|\t^*|_\infty - \nu\hat u \le |\hat \theta-\theta^*|_1 + \lambda |\theta^*|_2 + \nu|\hat\theta-\t^*|_\infty
\end{align*}
and the result follows since $|\hat\theta|_\infty \le \hat u$ and $|\hat\theta-\t^*|_\infty\le |\hat\theta-\t^*|_1$. Similar calculations yield the bound for $\hat u$.
\end{proof}

\begin{lemma}\label{lem5} Let $\hat \theta$ be the $\{\ell_1,\ell_2,\ell_\infty\}$-compensated MU-selector. Assume (A1)-(A3) and (A5). Then,
on a subset of the event of Lemma \ref{lem3} having probability at least $1-8\varepsilon$, we have
\begin{align*}
& \big|\fracn X^TX(\hat\theta-\theta^*)\big|_\infty \le  \mu_0 + \mu_1 |\hat\theta-\theta^*|_1 + \mu_2 |\theta^*|_2 + \mu_\infty|\theta^*|_\infty,
\end{align*}
where $\mu_0 = \tau+\delta_2+\delta_3$, $\mu_1 = 2\delta_1+\delta_4+\delta_5+b(\e)+\{(1+\nu)/\lambda\}\mu + \{(1+\lambda)/\nu\}\b$, $\mu_2 = \mu+\delta_1'$, $\mu_\infty = \b+b(\e)+\delta_4'+\delta_5$.
\end{lemma}
\noindent Note that $\mu_0$ and $\mu_2$ are of order $\sqrt{\fracn\log (c'p/\e)}$, and $\mu_1$ and $\mu_\infty$ are of order $\sqrt{\fracn\log (c'p/\e)}+b(\e)$.
%%%%%
%
\begin{proof}
Throughout the proof, we assume that we are on the event of probability
at least $1-6\varepsilon$ where inequalities (\ref{lem3.2}) -- (\ref{ineg1}) hold and $(\theta^*, |\theta^*|_2,|\t^*|_\infty)$ belongs to the feasible set of the minimization problem (\ref{conicNew}).  We
have \begin{align*} |\fracn X^TX\Delta|_{\infty} \le &|\fracn Z^T(Z\hat{\theta}-y)-{\widehat D}\hat{\theta}|_{\infty}+|(\fracn Z^TW-D)\hat{\theta}|_{\infty}\\+&|(\widehat
D-D)\hat{\theta}|_{\infty}+|\fracn Z^T\xi|_{\infty}+|\fracn W^TX\Delta|_{\infty}.
\end{align*}
Using the fact that $(\hat\theta, \hat t, \hat u)$ belongs to the feasible set of the minimization problem (\ref{conic}) together with (\ref{c2}), we obtain
\begin{align*}
 |\fracn Z^T(Z\hat{\theta}-y)-{\widehat D}\hat{\theta}|_{\infty} & \le \mu  \hat t + \b\hat u +\tau \\
 & \le \{(1+\nu)/\lambda\}\mu |\Delta|_1 + \mu |\theta^*|_2 + \{(1+\lambda)/\nu\}\b |\Delta|_1 + \b |\theta^*|_\infty +\tau.
\end{align*}
Using that $\hat\theta=\theta^*+\Delta$, Assumption (A5) together with (\ref{lem3.4}) yields that
\begin{align*}
|\fracn X^TX\Delta|_{\infty} \le & \{(1+\nu)/\lambda\}\mu |\Delta|_1 + \mu |\theta^*|_2 + \{(1+\lambda)/\nu\}\b |\Delta|_1 + \b |\theta^*|_\infty +\tau \\
& +|(\fracn Z^TW-D)\hat{\theta}|_{\infty}+|(\widehat
D-D)\hat{\theta}|_{\infty}+|\fracn Z^T\xi|_{\infty}+|\fracn W^TX\Delta|_{\infty}\\
\le & \{(1+\nu)/\lambda\}\mu |\Delta|_1 + \mu |\theta^*|_2 + \{(1+\lambda)/\nu\}\b |\Delta|_1 + \b |\theta^*|_\infty +\tau \\
& +|(\fracn Z^TW-D)\hat{\theta}|_{\infty}+b(\e)|\t^*|_\infty+b(\e)|\Delta|_1+\delta_2+\delta_3+|\fracn W^TX\Delta|_{\infty}.\\
\end{align*}
%$|\frac{1}{n}X^TX\Delta|_{\infty}$ is not greater than
%$$\mu|\hat\theta|_1+2\delta_2+2\delta_3+b|\hat\theta|_1+|(\frac{1}{n}Z^TW-D)\hat{\theta}|_{\infty}+|\frac{1}{n}W^TX\Delta|_{\infty}.
%$$
Now remark that $|(\fracn Z^TW-D)\hat{\theta}|_{\infty}\le |(\fracn Z^TW-D)\Delta|_{\infty} + |(\fracn Z^TW-D)\t^*|_{\infty}$. In view of Lemma \ref{lem3a} and (\ref{lem3.3}), on the initial event of probability at least $1-6\e$,
\begin{align}
&|(\fracn Z^TW-D)\theta^* |_{\infty}\nonumber\\
\leq&
|\fracn(W^TW-\text{Diag}\{W^TW\}) \theta^*|_{\infty}
+|(\fracn\text{Diag}\{W^TW\}-D)\theta^*|_{\infty} + \ |\fracn X^TW \theta^* |_{\infty}\nonumber\\
\leq& (\delta_4'+\delta_5)|\t^*|_\infty+\delta_1'|\t^*|_2. \label{c9}
\end{align}
Moreover, we have
\begin{multline*}|(\fracn Z^TW-D)\Delta |_{\infty}\leq
|\fracn Z^TW-D|_{\infty}|\Delta|_1\\
\le
\big(|\fracn(W^TW-\text{Diag}\{W^TW\})|_{\infty}
+|\fracn\text{Diag}\{W^TW\}-D|_{\infty}
+|\fracn X^TW|_{\infty}\big)|\Delta|_1.
\end{multline*}
Therefore,
\begin{equation} \label{c7}
|(\fracn Z^TW-D)\Delta |_{\infty}
\le (\delta_1+\delta_4+\delta_5)|\Delta|_1,
\end{equation}
with probability at least $1-8\e$ (since we intersect the initial event of probability at least $1-6\e$ with the event of probability at least $1-2\e$ where the bounds $\delta_1$ and $\delta_4$ hold for the corresponding terms).  Next, on the same event of probability at least $1-8\e$,
\begin{eqnarray}
 |\fracn W^TX\Delta|_{\infty}\leq |\fracn X^TW|_{\infty} |\Delta|_1 \leq \delta_1 |\Delta|_1.
\label{c8}
\end{eqnarray}
To complete the proof, it suffices to plug (\ref{c9}) -- (\ref{c8}) in the last inequality for $|\fracn X^TX\Delta |_{\infty}$ and to obtain
\begin{align*}
|\fracn X^TX\Delta|_{\infty} \le &  [2\delta_1+\delta_4+\delta_5+b(\e)+\{(1+\nu)/\lambda\}\mu + \{(1+\lambda)/\nu\}\b] |\Delta|_1 \\
& + \{\mu+\delta_1'\} |\theta^*|_2 +  \{\b+b(\e)+\delta_4'+\delta_5\} |\t^*|_\infty +\tau+\delta_2+\delta_3.
\end{align*}
\end{proof}

\section*{Acknowledgement}
The work of the third author is supported by GENES, and by the French National Research Agency (ANR) as part of Idex Grant ANR -11- IDEX-0003-02, Labex ECODEC (ANR - 11-LABEX-0047), and IPANEMA grant (ANR-13-BSH1-0004-02).

\end{document}